\documentclass[a4paper,11pt]{article}
\usepackage[margin=1in]{geometry}
\usepackage{amsmath,amsthm,amssymb}
\usepackage{tikz,enumitem}
\usepackage{comment}
\tikzset{main node/.style={circle,fill=blue!20,draw,minimum size=0.7cm,inner sep=0pt},
}

\usepackage{hyperref}
\hypersetup{
  colorlinks=true,
  citecolor=blue,
  linkcolor=black,
  filecolor=magenta,
  urlcolor=cyan
}
\usepackage{xcolor}

\newtheorem{theorem}{Theorem}[section]
\newtheorem{problem}[theorem]{Problem}

\newtheorem{corollary}[theorem]{Corollary}

\newtheorem{lemma}[theorem]{Lemma}

\newtheorem{claim}[theorem]{Claim}

\newtheoremstyle{case}{}{}{}{}{}{:}{ }{}
\theoremstyle{case}

\def\inst#1{$^{#1}$}

\newcommand{\thecolor}{black} 

\title{Structural description of (bull, house)-free graphs}

\author{
  Manoj Belavadi\inst{1}\footnote{email: \href{mailto:mbelavadi@isichennai.res.in}{mbelavadi@isichennai.res.in}}
  \and 
  Ch\'inh~T.~Ho\`ang\inst{2}\footnote{email: \href{mailto:choang@wlu.ca}{choang@wlu.ca}}
  }
\date{\today}
\begin{document}
\maketitle
	\begin{center}
		{\footnotesize 
			
			\inst{1} Indian Statistical Institute, Chennai Centre, Chennai 600029, India. 
			
			\inst{2} Department of Computer Science and Physics, Wilfrid Laurier University, \\ Waterloo, ON, Canada, N2L 3C5. 
    }
		
	\end{center}

\begin{abstract}
The bull is a graph consisting of a triangle and two pendant edges. The $P_5$ is the chordless path on five vertices. The house is the complement of a $P_5$. A graph is $k$-critical if it is $k$-chromatic but each of its proper induced subgraphs is $(k-1)$-colorable. It is known that the number of $k$-critical $P_5$-free graphs and bull-free graphs are infinite for large enough $k$.  We give a structural description of (bull, house)-free graphs and also (bull, $P_5$)-free graphs. Using these structural properties we prove that for any fixed $k$, the number of  $k$-critical (bull, $P_5$)-free graphs is finite. This improves on a result of Huang, Li and Xia [Critical ($P_5$, bull)-free graphs, Discrete Applied Mathematics 334 (2023) 15–25], and answers an open question of Beaton and Cameron [Vertex-critical graphs in subfamilies of ($P_4$+$\ell P_1$)-free graphs, arXiv:2604.06999v1 (2026)]. A graph $G$ is perfectly divisible if for each induced subgraph $H$ of $G$ with at least one edge, $V(H)$ can be partitioned into two sets $V_1, V_2$ such that every largest clique of $H$ contains a vertex in $V_i$ for $i = 1,2$. Chudnovsky and Sivaraman proved that ($P_5$, bull)-free graphs are perfectly divisible [Perfect divisibility and 2-divisibility, Journal of Graph Theory 90 (2019) 54–60]. Our structural result allows us to give a short proof of this theorem.
\end{abstract}
\section{Introduction}\label{sec:intro}
Graph coloring is a well known problem in computer science and
discrete mathematics. A coloring of a graph is an assignment of ``colors'' to its vertices so that each vertex receives one color, and two vertices are assigned different colors if they are adjacent.  If a graph $G$ can be colored with $k$ colors then we say that $G$ is $k$-{\it colorable}. The Graph Coloring problem is the problem of coloring a graph with the smallest number $k$ of colors; this number $k$ is the {\it chromatic number} of $G$ and denoted by $\chi(G)$.
A graph $G$ is $k$-{\it critical} if $\chi(G) = k$ but each of its proper induced subgraphs is $(k - 1)$-colorable. If a graph is not $k$-colorable then it contains an induced subgraph that is $k$-critical. Thus, the structure of $k$-critical graphs of a graph class ${\cal C}$ may explain why certain graphs in ${\cal C}$ are $k$-colorable. This line of research is encapsulated by the following problem. Following \cite{HuaLi2023}, we call it the Finiteness Problem.
\begin{problem}[Finiteness Problem]
    Given a \textcolor{\thecolor}{hereditary} class ${\cal C}$ of graphs and an integer $k$, are there a finite number of $k$-critical graphs in ${\cal C}$?
\end{problem}
The Finiteness Problem has an interesting consequence. If the answer to it is ``yes'' for a class ${\cal C}$ and an integer $k$, then there is a polynomial-time algorithm to decide if a graph in ${\cal C}$ is $(k-1)$-colorable. The purpose of this paper is to investigate the Finiteness Problem for $P_5$-free graphs. To do that, we will need to introduce a few definitions. 

\subsection{Definitions}
Let $G$ be a graph. Then $\overline{G}$ denotes the complement of $G$. We may also refer to $\overline{G}$ as co-$G$. A graph is {\it co}-bipartite if its complement is bipartite. $G$ is {\it co-connected} if $\overline{G}$ is connected.  If $A$ is a set of vertices of $V(G)$, then $G[A]$ denotes the subgraph of $G$ induced by $A$. For two vertices $a,b \in V(G)$, the {\it distance} from $a$ to $b$ is the number of edges in a shortest path from $a$ to $b$ (if such a path exists). A set $A \subset V(G)$ is {\it complete} (respectively, {\it anti-complete}) to $B \subset V(G)$ if there are all (respectively, no) edges between $A$ and $B$. A set $H$ of vertices of $V(G)$ is {\it homogeneous} if $1 < |H| <|V(G)|$ and every vertex in $V(G) \setminus H$ is complete or anti-complete to $H$.  A graph is {\it prime} if it does not contain a homogeneous set. Note that prime graphs are connected and co-connected. For a vertex $v \in V(G)$, $N(v)$ denotes the set of vertices adjacent to $v$ in $G$. For a set $X \subseteq V(G)$, $N(X)$ denotes the set of vertices in $V(G) \setminus X$ that are adjacent to some vertex in $X$.

$P_k$ (respectively, $C_k$) denotes the chordless path (respectively, cycle) with $k$ vertices. Let ${\cal F}$ be a set of graphs, we say that $G$ is ${\cal F}$-free if no induced subgraph of $G$ is isomorphic to a graph in ${\cal F}$. A {\it hole} is the graph $C_k$ with $k \geq 5$. A hole is {\it odd} if it has an odd number of vertices. 
If $A$ and $B$ are two disjoint graphs, then $A + B$ denotes the union of $A$ and $B$. A {\it bull} (see Figure \ref{fig:5-vertex-graphs}) is the graph obtained from a triangle with two disjoint pendant edges.  The bull has been much studied. 
\textcolor{\thecolor}{
We let $\omega(G)$ denote the number of vertices in a largest clique of $G$. 
}
A graph $G$ is  {\it perfectly divisible }if for all induced subgraphs $H$ of $G$ with at least one edge, $V(H)$ can be partitioned into two sets $A,B$ such that $\omega(G[A]) < \omega(G)$ and $\omega(G[B]) < \omega(G)$. A graph $G$ is {\it perfect} if each induced subgraph $H$ of $G$ satisfies $\chi(H) = \omega(H)$. Note that the only $k$-critical perfect graph is the clique on $k$ vertices. Chudnovsky, Robertson, Seymour, and Thomas \cite{ChuRob2006} proved that a graph is perfect if and only if it does not contain an induced subgraph isomorphic to an odd hole or its complement. 

\subsection{Background and results}

Recently, much research has been done on determining the $k$-critical graphs of the class of $P_5$-free graphs.  Finding the chromatic number of a $P_5$-free graph is
NP-hard \cite{KraKra2001}, but for every fixed $k$, the problem of coloring a graph with $k$ colors admits a polynomial-time
algorithm \cite{HoaKam2008, HoaKam2010}. Bruce, Ho\`ang and Sawada \cite{BruHoa2009} showed that the number of $4$-critical $P_5$-free graphs is finite. However, for any $k \geq 5$, the number of $5$-critical $P_5$-free graphs is infinite \cite{HoaMoo2015, HoaMoo2013}. Because of this result, many researchers have investigated the Finiteness Problem of $k$-critical $(P_5, H)$-free graphs, where $H$ is a certain graph. Dhaliwal et al.  \cite{DhaHam2017} showed that for each $k$, the number of $k$-critical ($P_5, \overline{P}_5$)-free graphs is finite (see also \cite{HoaLaz2015}). Cameron and Ho\`ang \cite{CamHoa2024} showed that the number of $k$-critical ($P_5, C_5$)-free graphs, for $k \geq 5$ is infinite. For a 5-vertex graph $H$, we note the work in \cite{CaiGoe2023} for ($P_5$, gem)-free graphs and $(P_5, \overline{P_2 + P_3})$-\textcolor{blue}{free} graphs, in \cite{HuaLi2023} for ($P_5$, chair)-free graphs for $k=5$. We are particularly interested in the case $H$ is the bull graph.  The bull has been much studied. Chudnovsky and Safra \cite{ChuSaf2008} showed that a bull-free graph on $n$ vertices has $\omega(G) \geq n^{\frac{1}{4}}$ or $\omega(\overline{G}) \geq n^{\frac{1}{4}}$; thus solving the well known Erd\H os–Hajnal conjecture (\cite{ErdHaj1989}) for bull-free graphs. Chudnovsky and Sivaraman \cite{ChuSiv2019} proved that ($P_5$, bull)-free graphs are perfectly divisible. In this context, we mention the following two conjectures of Ho\`ang: (i) Odd hole-free graphs are perfectly divisible \cite{Hoa2026}, and (ii) $P_5$-free graphs are perfectly divisible \cite{Hoa2018}.

\begin{figure}
\newcommand{\ang}{17}
\newcommand{\sep}{15}
\begin{tikzpicture}[scale=1]
	\tikzstyle{every node}=[font=\small]
\newcommand{\size}{1}
\newcommand{\offset}{2}

\newcommand{\bull}{1}{
	\path ( \size * 2 + \offset  ,0) coordinate (g1);
	\path (g1) +(\size , \size ) node (g1_1){}; 
	\path (g1) +(\size * 2,  \size) node	(g1_2){}; %
	\path (g1) +(\size , \size * 2.0 ) node (g1_3){}; 
    \path (g1) +(\size * 2,  \size *2) node	(g1_5){};
	  \path (g1) +(1.5 * \size, 0) node (g1_4){}; 
	 
	\foreach \Point in {(g1_1),(g1_2),(g1_3),(g1_4), (g1_5)}{
		\node at \Point {\textbullet};
	}
	\draw[thick]  (g1_1) -- (g1_2) (g1_1) -- (g1_4) (g1_1) -- (g1_3)
	(g1_2) -- (g1_5) (g1_2) -- (g1_4) 
	;
	\path (g1) ++(\size * 1.5  ,-\size ) node[draw=none,fill=none] { {\large Bull}};
}

\newcommand{\g3}{3}{
	\path( 8 * \size, 0) coordinate(g3);
	\path(g3) + (\size , 0) node(g3_1){}; 
	\path(g3) +(\size, \size) node(g3_2){};
	\path(g3) +(  \size, \size * 2) node(g3_3){};
	\path(g3) +(\size *2, \size) node(g3_4){};
	\path(g3) +(2* \size, 0) node(g3_5){};
	
	\foreach \Point in {(g3_1),(g3_2),(g3_3),(g3_4), (g3_5)}{
		\node at \Point {\textbullet};
	}
	\draw[thick]   (g3_1) -- (g3_2)
	(g3_2) -- (g3_3)
	(g3_4) -- (g3_2)
	(g3_4) -- (g3_5);
	\path (g3) ++(\size  *1.5 ,-\size ) node[draw=none,fill=none] { {\large Chair}};
	
}

\newcommand{\g4}{4} {
	\path( 13 * \size, 0) coordinate(g4);
	\path(g4) + (0,0)  node(g4_1){}; 
	\path(g4) + ( \size , 0) node(g4_2){};  
	\path(g4) + ( \size * 2 , 0) node(g4_3){};  
	\path(g4) + ( \size * 3 , 0) node(g4_4){};  
	\path(g4) + ( \size * 1.5 , \size) node(g4_5){};  
	
	\foreach \Point in {(g4_1),(g4_2), (g4_3), (g4_4), (g4_5)}{
		\node at \Point {\textbullet};
	}
	\draw   (g4_1) -- (g4_2)
	(g4_2) -- (g4_3) 
	(g4_2) -- (g4_4) 
	(g4_3) -- (g4_5) 
	(g4_4) -- (g4_5) (g4_1) -- (g4_5) (g4_2) -- (g4_5);
	\path (g4) ++(\size  * 1.5,-\size ) node[draw=none,fill=none] { {\large Gem}};
}

 \end{tikzpicture}
\caption{Bull, Chair and Gem}\label{fig:5-vertex-graphs}
\end{figure}
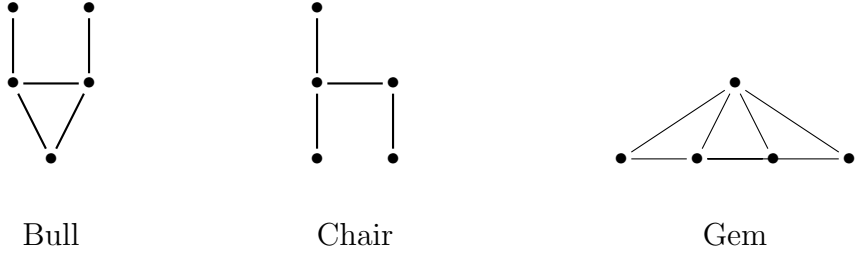

Huang, Li, and Xia \cite{HuaLi2023b} showed that there are finitely many 5-vertex-critical ($P_5$, bull)-free graphs. It is this result that motivates us to study the structure of ($P_5$, bull)-free graphs. Our result implies the following

\begin{theorem}\label{thm:critical}
    For every $k \geq 1$, the number \textcolor{blue}{of} $k$-critical ($P_5$, bull)-free graphs is finite. 
\end{theorem}

On the way to proving Theorem~\ref{thm:critical}, we will obtain a new proof of the result of Chudnovsky and Sivaraman on ($P_5$, bull)-free graphs. It will be more convenient for us to describe the structure of the complements of $P_5$-free graphs.  A {\it house} is the complement of the $P_5$. A {\it triangle} is the clique on three vertices.

\begin{theorem}\label{thm:(bull,house)-free}
    Let $G$ be a connected (bull, house)-free graph, then
    \begin{itemize}[noitemsep, topsep=0pt]
        \item[$(i)$] $G$ contains a homogeneous set, or
        \item[$(ii)$] $G$ is triangle-free, or
        \item[$(iii)$] $G$ is co-bipartite.
    \end{itemize}
\end{theorem}

Karthick, Maffray, and Pastor \cite{KarMaf2019} proved the following
\begin{theorem}[Theorem 5.2 in \cite{KarMaf2019}]\label{thm:KarMaf}
    Let $G$ be a connected (bull, house)-free graph, then
    \begin{itemize}[noitemsep, topsep=0pt]
        \item[$(i)$] $G$ contains a homogeneous set, or
        \item[$(ii)$] $G$ is triangle-free, or
        \item[$(iii)$] $G$ is ($P_5, C_5$)-free.
    \end{itemize}
\end{theorem}
Theorem \ref{thm:KarMaf} combined with the result in \cite{ChvHoa1987} also implies Theorem \ref{thm:critical}. However, our Theorem \ref{thm:(bull,house)-free} gives a stronger description of the structure of (bull, house)-free graphs.

In section \ref{sec:structure}, we will prove Theorem \ref{thm:(bull,house)-free}. In section \ref{sec:critical}, 
we will prove Theorem \ref{thm:critical}. In section \ref{sec:divisibility}, we give a new and short proof of the result of \cite{ChuSiv2019}. In section \ref{sec:conclusions}, we give a conclusion and discuss open problems related to our work.  


\section{Structure of (bull, house)-free graphs}\label{sec:structure}                                                     
In this section, we give a proof of Theorem \ref{thm:(bull,house)-free}.
We write $x \sim y$ to mean $x$ is adjacent to $y$, and $x \not\sim y$ to mean $x$ is not adjacent to $y$
\begin{lemma}\label{lem:split-vertex}
    Let $G$ be a connected (bull, house)-free graph. Suppose $v$ is a vertex of $G$, such that $G[N(v)]$ has two components, then $G$ is triangle-free or $G$ has a  homogeneous set \textcolor{\thecolor}{$H$ such that $G[H]$ is connected}.
\end{lemma}
\begin{proof}
    Assume the hypothesis. We will prove by contradiction. We may assume that $G$ has no homogeneous set. Let $a,b$ be two vertices of $G$, define $dist(a,b)$ to be the number of edges in a shortest path from $a$ to $b$. If such a path does not exist ($G$ is not connected), then we let $dist(a,b) = \infty$.
    Let $N_0 := \{v\}$. For $i\ge 1$, let $N_{i}(v) := \{u\in V(G) | ~dist(u, v)=i\}$. 
    We are going to prove that for each $i \ge 0$ and for every $u\in N_{i}$, $N(u)$ is a stable set. The proof is by induction on $i$.

    For $i = 0$, suppose that $G[N(v)]$ is not a stable set.  Let $A$ induce a component of $G[N(v)]$ with an edge. By assumption on $G$, $A$ cannot be a homogeneous set. Thus, there exist $x_1$, $x_2$ in $A$ such that $x_1x_2 \in E(G)$ and there exists $y\notin N(v)$ such that $y \sim x_2$  but  $y \not\sim x_1$. Since $G[N(v)]$ has at least two components, there exists $z$ in $N(v)\setminus A$. 
    Then $\{z,v,x_1,x_2,y\}$ induces either a bull, or a house.

    For $i > 0$, consider a vertex $u \in N_i$. Define $N'_j = N(u) \cap N_j$ for $j = i-1, i , i+1$ (if $N_{i+1} $  exists). We know that $N'_{i-1}$ is a stable set and there are no edges between $N'_{i-1}$ and $N'_{i}$ because each $x \in N'_{i-1}$ is such that $N(x)$ is a stable set (by the induction hypothesis). Suppose $N(u)$ contains an edge. Let $A$ be the component of $G[N(u)]$ that contains this edge. Note that $A \subseteq N'_i \cup N'_{i+1}$. Since $A$ cannot be a homogeneous set, there are vertices $a,b,z$ with  $z \notin N(u), a,b \in N'_i \cup N'_{i+1} , z \sim a, a \sim b , z \not\sim b$.    Let $u'$  be a neighbour of $u$ in $N_{i-1}$. By the induction hypothesis, we have $u'a, u'b \notin E(G)$.  But now $\{u', u, a, b, z\}$ induces a bull or a house. 
\end{proof}


\begin{lemma}\label{lem:hole}
    Let $G$ be a connected (bull, house)-free graph. Suppose $G$ contains a hole with at least five vertices, then $G$ is triangle-free, or $G$ contains a  homogeneous set $H$ such that $G[H]$ is connected.
\end{lemma}
\begin{proof}
    Assume $G$ contains a hole induced by $X = \{x_1, x_2, \ldots, x_{\ell}\}$ such that $x_{i}x_{i+1}\in E(G)$ where $i$ is taken modulo $ \ell \ge 5$. Let $U$ be the set of vertices
    adjacent to every vertex in $X$. Suppose $U$ is nonempty. We first prove that $U$ is complete to $N(X)\setminus U$. Let $y\in N(X)\setminus U$, then there exists $j$ modulo $l$ such that $y \sim x_j$  but  $y \not\sim x_{j-1}$. Suppose $a\in U$ is not adjacent to $y$, then $\{y, x_{j}, x_{j-1}, a, x_{j+2}\}$ induces either a bull or a house, a contradiction. So $U$ is complete to $N(X)\setminus U$.
    
    Write $N_0 = X$, $N_1 = N(X)\setminus U$. Define, for $i \geq 2$,  $N_i = \{ v \in V(G) \setminus (U \cup (\cup_{k=0)}^{k=i-1} N_k)) \; | \; v$ has a neighbor in $N_{i-1} \}$.

    We will prove that  
\begin{equation}\label{eq:level}
	\text{\parbox{.85\textwidth}{$U$ is complete to $\cup_{k=0}^{i} N_k$ for any index $i \geq 0$}}
\end{equation}
    We prove (\ref{eq:level}) by induction on $i$. We know (\ref{eq:level}) holds for $i = 0,1$. Consider the case $i=2$. Suppose that there is a vertex $a \in U$ that is not adjacent to a vertex $z \in N_2$. Consider a neighbor $y \in N_1$ of $z$. We know $a \sim y$. There is an index $j$ such that  $ y \sim x_j$ and  $ y \not\sim x_{j-1}$. Vertex $y$ is adjacent to $x_{j-2}$, for otherwise $\{ x_j, y, a, x_{j-2}, z \}$ induces a bull. 


Vertex $y$ is not adjacent to $x_{j+1}$, for otherwise $\{x_{j+1}, y, x_j, x_{j-1}, x_{j-2}  \}$ induces a house. But now $\{ x_{j-2}, y, a, x_{j+1}, z \}$ induces a bull. So $U$ is complete to $N_2$. Now, by the induction hypothesis, we may assume $U$ is complete to $\cup_{k=0}^{k=i-1} N_k$. For $i \geq 3$, suppose there is a vertex $a \in U$ that is not adjacent to a vertex $z \in N_i$. Consider vertices $b_{i-1} \in N_{i-1}$, $b_{i-2} \in N_{i-2}$ such that $b_{i-1} \sim b_{i-2}, z \sim b_{i-1}$, $z \not\sim b_{i-2}$. Consider a vertex $r \in X$ that is not adjacent to $b_{i-2}$. Now $\{b_{i-2}, b_{i-1}, a, r, z \}$ induces a bull. So (\ref{eq:level}) holds. 
    
    Let $R = V(G) \setminus U \cup (\cup_{k=0}^{i} N_k)$. Then $R$ is anti-complete to  $\cup_{k=0}^{i} N_k$. Thus, $\cup_{k=0}^{i} N_k$ is a connected homogeneous set. So we may assume $U$ is empty. 

    \begin{equation}\label{eq:three-consecutive}
	\text{\parbox{.85\textwidth}{Let $y$ be a vertex in $V(G) \setminus X$ with $y \sim x_i, y \sim x_{i+1}$. Then we have $N(y)\cap X = \{x_i,x_{i+1},x_{i+2}\}$ or $N(y) \cap X = \{x_{i-1}, x_i,x_{i+1}\}$.}}
    \end{equation}

 For simplicity, we will let $i=2$. Let $y$ be a vertex in $V(G) \setminus X$ with $y \sim x_2, y \sim x_{3}$.  Suppose that (\ref{eq:three-consecutive}) does not hold for $y$.
 Since $\{y,x_2, x_3, x_{4}, x_1\}$ does not induce a bull, $y$ must be adjacent to $x_1$ or $x_4$. Assume $y \sim x_{1}$. Let $j$ be the smallest index such that $yx_j\notin E(G)$ (since $U = \emptyset$, $j$ exists) and let $k \geq 4 $ be the largest index such that $y x_k\in E(G)$ ($k$ exists since (\ref{eq:three-consecutive}) fails for $y$). If $j = 4$, then $\{ x_2, x_3, y, x_k, x_4 \}$ induces a bull or a house. So we have $j > 4$. But now $\{ x_{j-2}  , x_{j-1} , y, x_1, x_j\}$ induces a bull or a house. So (\ref{eq:three-consecutive}) holds.

Suppose $N(x_2)\cap N(x_3)$ is empty, then $G[N(x_2)]$ is disconnected, and the result follows from Lemma \ref{lem:split-vertex}. So assume $N(x_2)\cap N(x_3)$ is non-empty. Note that, from (\ref{eq:three-consecutive}), the set $N(x_2)\cap N(x_3)$ is contained in $A_2 \cup A_3$, where $A_2$ = $\{v\in V(G)| N(v)\cap X = \{x_1, x_2, x_3\} \text{ or } N(v)\cap X = \{x_1, x_3\} \}$ and $A_3$ = $\{v\in V(G)| N(v)\cap X = \{x_2, x_3, x_4\} \text{ or } N(v)\cap X = \{x_2, x_4\}\}$. Note that $x_2\in A_2$ and $x_3\in A_3$.

    We may assume there is a vertex $y\in A_2$ such that $yx_2\in E(G)$. Let $B$ induce the component in $G[A_2]$ that contains the edge $yx_2$. We claim that $B$ is a homogeneous set. Suppose not, there exist adjacent vertices $b_1$, $b_2$ in $B$ such that there exists a vertex $z\notin B$ adjacent to $b_2$ but not to $b_1$. If $z$ is not adjacent to $x_3$ (resp. $x_1$), then $\{  b_1, b_2, x_3, x_4, z\}$ (resp. $\{ b_1, b_2,x_1, x_\ell,  z\}$) induces either a bull or a house. Hence $z\in N(x_1)\cap N(x_3)$.  We claim that $z\in A_2$. Suppose not, then $zx_{i}\in E(G)$ for some $i\notin \{1,2,3\}$. If $i=4$, then $\{x_4, z ,x_3, b_1, x_1\}$ induces a house, and if $i\neq 4$, then $\{ b_2, z,  x_3, x_4, x_i\}$ induces either a bull or a house. So $z\in A_2$ and this contradicts $z\notin B$. Therefore, $B$ is a connected homogeneous set.
\end{proof}
    
    
    %

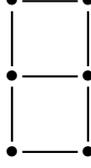
\begin{figure}
\newcommand{\ang}{17}
\newcommand{\sep}{15}
\begin{tikzpicture}[scale=1]
	\tikzstyle{every node}=[font=\small]
\newcommand{\size}{1}
\pgfmathsetmacro{\offset}{7.5}
\newcommand{\bull}{1}{
	\path ( 0  , 0) coordinate (g1);
	\path (g1) +(\offset , 0 ) node (g1_3){}; 
	\path (g1) +(\offset , \size) node	(g1_2){}; %
	\path (g1) +(\offset , \size *2 ) node (g1_1){}; 
    \path (g1) +(\offset + \size, 0) node	(g1_4){};
	  \path (g1) +(\offset + \size, \size) node (g1_5){}; 
    \path (g1) +(\offset + \size, \size * 2) node (g1_6){}; 
	 
	\foreach \Point in {(g1_1),(g1_2),(g1_3),(g1_4), (g1_5), (g1_6) }{
		\node at \Point {\textbullet};
	}
	\draw[thick]  (g1_1) -- (g1_2) (g1_1) -- (g1_6) (g1_2) -- (g1_3)
	(g1_3) -- (g1_4) (g1_4) -- (g1_5) (g1_5) -- (g1_6) (g1_2) -- (g1_5)
	;
	\path (g1) ++(\size * 1.5  ,-\size ) node[draw=none,fill=none] { {}};
}

\end{tikzpicture}
\caption{The domino}\label{fig:domino}
\end{figure}

The domino is the graph shown in Figure \ref{fig:domino}.
\begin{lemma}\label{lem:domino}
    Let $G$ be a connected (bull, house, hole)-free graph that contains a domino. Then $G$ is triangle-free or $G$ contains a \textcolor{\thecolor}{connected} homogeneous set.
\end{lemma}
\noindent {\it Proof.}
    Let $G$ be a connected (bull, house)-free graph that contains a domino induced by $D$ = $\{a_1, a_2, \dots, a_6\}$ such that $a_ia_{i+1}\in E(G)$ with $i$ taken modulo 6, and $a_3a_6\in E(G)$. Define sets $A_i$ = $\{v\in V(G)|~ N(v)\cap (D\setminus \{a_i\}) = N(a_i)\cap D\}$ for $i$ modulo 6. Let $U$ be the set of all vertices complete to $D$ and let $R$ be the set of all vertices anti-complete to $D$. All subscripts are taken modulo 6.

Next, we will establish the following
    \begin{claim}\label{Claim1}
        For any $x\in N(D)$, if $x$ is adjacent to two adjacent vertices of $D$, say $a_i$, $a_j$, then $x\in A_i\cup A_j$ or $x\in U$.
    \end{claim}
We will first establish the following statement.   
\begin{equation}\label{eq:see-four}
	\text{\parbox{.85\textwidth}{
    Let $x\in N(D)$. If $x$ is complete to $\{ a_1, a_2, a_3, a_6\}$, then $x \in U$. Symmetrically, if $x$ is complete to $\{ a_3, a_4, a_5, a_6\}$, then $x \in U$.
    }}
\end{equation}
    
    Suppose $x$ is complete to $\{ a_1, a_2, a_3, a_6\}$, but $x \not\in U$. Without loss of generality, we may assume $x$ is not adjacent to $a_4$. So $x a_5 \in E(G)$, for otherwise $\{ x, a_3, a_4, a_5, a_6 \}$ induces a house. But now $\{ x, a_2, a_3, a_4, a_5\}$ induces a house. The second part of (\ref{eq:see-four}) follows by symmetry. 

    We continue the proof of Claim \ref{Claim1}. Up to symmetry we may assume (i) $x\in N(a_1)\cap N(a_2)$, or (ii) \textcolor{\thecolor}{ $x\in (N(a_2)\cap N(a_3))$}, or (iii) $x\in (N(a_3)\cap N(a_6))$. 
    
   Consider  case (i), that is, $x\in N(a_1) \cap N(a_2)$. By (\ref{eq:see-four}) we may assume $x$ is nonadjacent to at least one vertex in $\{ a_3, a_6 \}$. If $x a_3, x a_6 \notin E(G)$, then $\{ x, a_1, a_2, a_3, a_6\}$ induces a house. So $x$ is adjacent to exactly one vertex in $\{ a_3, a_6 \}$. Without loss of generality, assume $x a_3 \notin E(G)$ (and $x a_6 \in E(G)$). We have $x a_5 \notin E(G)$, for otherwise $\{   a_1, a_2, x, a_5, a_3 \}$ induces a bull, and $x a_4 \notin E(G)$, for otherwise $\{ a_1, x, a_2, a_3, a_4\}$ induces a house. So $x \in A_1$. If $x a_6 \notin E(G)$, then we have by symmetry, $x \in A_2$ and we are done.

Now, consider case (ii), that is, $x\in (N(a_2)\cap N(a_3))$. Suppose $x a_1 \in E(G)$. Then by case (i), we know $x \in U \cup (A_1 \cup A_2)$. But we have $x \not\in A_1$ because $x a_3 \in E(G)$. So $x \in U \cup A_2$ and we are done. So we know $x a_1 \not\in E(G)$. We have the following implications: 
\begin{itemize} 
\item $x a_6 \in E(G)$, for otherwise $\{ x, a_2, a_3, a_6, a_1 \}$ induces a house,
\item $ x a_4 \in E(G)$, for otherwise, $\{x, a_3, a_6, a_1, a_4 \}$ induces a bull, 
\item $x a_5 \not\in E(G)$, for otherwise $x \in U$ by (\ref{eq:see-four}), and this would contradict $x a_1 \not\in E(G)$.
\end{itemize} 
But now $x \in A_3$, and we are done. 

    Now, consider case (iii), that is, $x\in (N(a_3)\cap N(a_6))$. Vertex $x$ must be adjacent to at least one vertex of $\{ a_4, a_5\}$, for otherwise $\{ x, a_3, a_6, a_5, a_4\}$ induces a house. 
    Vertex $x$ must be nonadjacent to at least one vertex of $\{ a_4, a_5\}$, for otherwise by (\ref{eq:see-four}), we have $x \in U$. Suppose $x a_4 \notin E(G)$ (and so $x a_5 \in E(G)$). We have $x a_1 \in E(G)$, for otherwise, $\{x, a_3, a_6,  a_1, a_4 \}$ induces a bull. We have $x a_2 \notin E(G)$, for otherwise $\{ x, a_2, a_3, a_4, a_5\}$ induces a house. But now $x \in A_6$. By a symmetrical argument, we see that if $x a_4 \in E(G)$ (and $ x a_5 \notin E(G)$), then $x \in A_3$.
    This completes the proof of Claim \ref{Claim1}.

    \begin{claim}\label{Claim2}
        Suppose, for some $i\in \{1,2,\ldots,6\}$, there are two adjacent vertices in $A_i$, then $G$ has a \textcolor{\thecolor}{connected} homogeneous set.
    \end{claim} 

    By symmetry of the domino, we only prove this Claim for $i=1$ and $i=3$.
    
    \vspace{5pt}
    \noindent
    For $i=1$, let $C$ induce a component of $G[A_1]$ with an edge. We claim that $C$ is a homogeneous set. If not, there exist two adjacent vertices $a$, $a' \in A_1$ and a vertex $x \notin A_1$, such that $xa'\in E(G)$ but $xa\notin E(G)$. Suppose $xa_6\notin E(G)$ (resp. $xa_2\notin E(G)$), then $\{a,a',a_5,a_6,x\}$ (resp. $\{a, a', a_2,a_3, x\}$) would induce either a bull, or a house. So $xa_6, xa_2\in E(G)$. We have $xa_5\notin E(G)$ (for otherwise $\{ a_5, a_6, x, a_2, a\}$ would induce a house), $xa_4\notin E(G)$ (for otherwise $\{a', x, a_6, a_5, a_4\}$ would induce a house), and $xa_3\notin E(G)$ (for otherwise $\{x, a_3, a_4, a_5, a_6\}$ would induce a house). This contradicts $x\notin A_1$.

    \vspace{5pt}
    \noindent 
    For $i=3$, let $C^{'}$ induce a component of $G[A_3]$ with an edge. We claim that $C^{'}$ is a homogeneous set. If not, there exist two adjacent vertices $c$, $c' \in A_3$ and a vertex $x \notin A_3$, such that $xc'\in E(G)$ but $xc\notin E(G)$. If $xa_2\notin E(G)$ (resp. $xa_4\notin E(G)$), then $\{c, c', a_2, a_1, x\}$ (resp. $\{ c', c, a_4, a_5, x\}$) induces either a bull or a house. So we have $xa_2, xa_4 \in E(G)$. Suppose $xa_1 \in E(G)$ (resp. $xa_5\in E(G)$), then $\{a_1,a_2,x, a_4, c\}$ (resp. $\{a_5, x, a_4, c , a_2\}$) would induce a house. So we have $xa_1$, $xa_5 \not\in E(G)$. Now if $xa_6\notin E(G)$, then $\{ x, c', a_2, a_1, a_6\}$ would induce a house. Then $x$ must be in $A_3$, a contradiction. Claim \ref{Claim2} is established.

\vspace{8pt}
\noindent
    From Claims \ref{Claim2} and \ref{Claim1}, we assume that the sets $A_i$'s are independent for all $i\in \{1,2,\ldots,6\}$ and that no vertex in $N(D)\setminus U$ is adjacent to two adjacent vertices of $D$. Suppose $U$ is empty, then there is no vertex in $N(a_1)$ which is adjacent to $a_2$. This implies that $G[N(a_1)]$ is disconnected and the result follows from Lemma \ref{lem:split-vertex}. So we assume that $U$ is nonempty. We first prove that $U$ is complete to every vertex in $N(a_i)\setminus U$, for all $i$ modulo 6.

\begin{claim}\label{Claim3}
        For any $x\in N(D)\setminus U$, $x$ is complete to $U$. Moreover, every vertex in $N(x)\cap R$ is complete to $U$.
\end{claim}

Since $x\notin U$, there exists an $i$ modulo $6$ such that $xa_i\in E(G)$ but $xa_{i+1}\notin E(G)$. Let $x$ be nonadjacent to some vertex $b\in U$. When $i\notin \{3,6\}$, $\{a_{i+1}, a_i, b,a_{i-3}, x\}$ induces either a bull or a house. When $i\in \{3,6\}$, $\{a_{i+1}, a_i, b,a_{i-2}, x\}$ induces either a bull or a house. So we conclude that each vertex in $N(D)\setminus U$ is complete to $U$. Suppose $y\in N(x)\cap R$ is nonadjacent to some vertex $b'\in U$. When $i\notin \{3,6\}$, since $x$ cannot be adjacent to both $a_{i+3}$ and $a_{i+4}$ (by Claim \ref{Claim1}), we see that $\{a_{i}, x,  b', a_{i+3}, y\}$ or $\{a_{i}, x,  b', a_{i+4}, y\}$ induces a bull. When $i\in \{3,6\}$, if $x \notin N(a_{i+4})$, then $\{a_i, b', x,y, a_{i+4}\}$ induces a bull, and if $x\in N(a_{i+4})$, then $\{a_{i+4},x, b',a_{i+1}, y\}$ induces a bull.  
Claim \ref{Claim3} is established. 

We continue the proof of the Lemma. 
Let $C^{*}$ induce the component in $G[V(G)\setminus U]$ that contains $D$. We claim that $C^{*}$ is complete to $U$. 
Define $N_0 = D, N_1 = N(D) \setminus U$ and for $i \geq 2, N_i = \{ z \in C^* \setminus (\cup_{k = 0}^{i-1} N_k )|$ $z$ has a neighbor in $N_{i-1}$\}. We have $C^* = \cup_{k = 0}^{m} N_k$, where $m$ is the largest subscript such that $N_m \not= \emptyset$. By Claim \ref{Claim3}, we know that $U$ is complete to $N_0 \cup N_1 \cup N_2$. Suppose that some vertex $b \in U$ is not adjacent to some vertices in $C^*$. Let $i$ be the smallest subscript such that $b$ is complete to $\cup_{k = 0}^{i} N_k$ but not complete to $N_{i+1}$. We know $i \geq 3$. Let $y$ be a vertex in $N_{i+1}$ that is not adjacent to $b$. By definition there are vertices $y_i \in N_i, y_{i-1} \in N_{i-1}$ such that $y y_{i}, y_{i-1} y_{i} \in E(G)$. By definition, there is a vertex $x_j \in D$ such that $x_j y_{i-1} \not\in E(G)$. Now the set $\{y, y_i, y_{i-1}, x_j, b\}$ induces a bull.
Therefore, $U$ is complete to $C^{*}$. Hence, $C^{*}$ is a homogeneous set of $G$. \qed

A graph is {\it HHD}-free if it does not contain an induced subgraph isomorphic to a house, hole, or domino. We will need the following result known in the literature (a vertex $x$ is {\it simplicial} if $N(x)$ induces a clique). 
\begin{theorem}[\cite{HoaKho1988}]\label{thm:HHD-free}
    Let $G$ be an HHD-free graph; then at least one of the following conditions must be satisfied:
    \begin{itemize}[noitemsep, topsep=0pt]
        \item[(i)] $G$ is a clique;
        \item[(ii)] $G$ contains a homogeneous set that induces a connected subgraph in the complement $\overline{G}$;
        \item[(iii)] $G$ contains two nonadjacent simplicial vertices.
    \end{itemize}
\end{theorem}

Now we can prove our structural theorem for (bull, house)-free graphs. 

\noindent {\it Proof of Theorem \ref{thm:(bull,house)-free}}. We prove by contradiction. Let $G$ be a connected (bull, house)-free graph. Assume that $G$ does not satisfy (i)-(ii). From Lemma \ref{lem:hole}, we assume that $G$ is hole-free and from Lemma \ref{lem:domino}, we assume that $G$ is domino-free. From Theorem \ref{thm:HHD-free}, $G$ contains nonadjacent simplicial vertices, say $v$ and $v'$. Let $K = N(v)$, $W = N(K) \setminus \{v\}$, and  $R = V(G) \setminus (\{v\}\cup K\cup W)$. Since $G$ is connected,  we may assume $W \not = \emptyset$, for otherwise $K$ is a homogeneous set (if $|K| \geq 2$), or $G$ is triangle-free (if $|K| = 1$). 

We will establish a number of claims, as follows.
\begin{equation}\label{eq:K-big}
	\text{\parbox{.85\textwidth}{
    $|K| \ge 2$.
    }}
\end{equation}
Suppose $K$ contains a single vertex $k$. Since $G$ is connected, vertex $k$ must have a neighbor in $W$. Then $G[N(k)]$ is disconnected and the result follows from Lemma \ref{lem:split-vertex}. Thus, (\ref{eq:K-big}) holds.
\begin{equation}\label{eq:K-W-complete}
	\text{\parbox{.85\textwidth}{
     There is a vertex, say $v_1$, in $K$ complete to $W$, and there is a vertex, say $w_1$, in $W$ complete to $K$.
    }}
\end{equation}
If there is a vertex $v'$ in $K$ that has no neighbor in $W$, then $\{v, v'\}$ is a homogeneous set. So we assume that every vertex in $K$ has a neighbor in $W$. We claim that for any two vertices $u_1$, $u_2\in K$, we have $N(u_1)\subseteq N(u_2)$ or $N(u_2)\subseteq N(u_1)$. If not, there exist vertices $w$, $w' \in W$ such that $w\in N(u_1)\setminus N(u_2)$ and $w'\in N(u_2)\setminus N(u_1)$. We see that $\{v, u_1, u_2, w', w\}$ would induce either a bull or a house. Let $v_1$ be a vertex in $K$ with the most neighbors in $W$. Then $v_1$ is adjacent to every vertex in $W$. Similarly, we can see that for any two vertices $w$, $w' \in W$, $N(w)\cap K \subseteq N(w')\cap K$ or $N(w)^{'}\cap K \subseteq N(w)\cap K$. So there exists a vertex, say $w_{1}$, in $W$ which is adjacent to every vertex in $K$. Thus (\ref{eq:K-W-complete}) holds. 
\begin{equation}\label{eq:W-R}
	\text{\parbox{.85\textwidth}{
     For any $x$, $y \in W$, if $xy\in E(G)$, then $N(x)\cap R$ = $N(y)\cap R$.
    }}
\end{equation}

Suppose for some adjacent vertices $x$, $y\in W$, we have $N(x)\cap R \neq N(y)\cap R$. Without loss of generality, assume $r\in N(y)\cap R$ and $rx \notin E(G)$. Then $\{x, y,  v_1, v , r\}$ induces a bull. So (\ref{eq:W-R}) holds. 
%
%
\begin{equation}\label{eq:y-dom-x-in-K}
	\text{\parbox{.85\textwidth}{
     For any two nonadjacent vertices $x$, $y \in W$, if $y$ is adjacent to a vertex $r \in R$, then $N(y) \cap K \subseteq N(x) \cap K$.
    }}
\end{equation}
Suppose for some nonadjacent vertices $x$, $y \in W$, $y$ is adjacent to a vertex $r \in R$ but $N(y) \cap K \nsubseteq N(x) \cap K$. Then there is a vertex $z\in N(y)\cap K$ which is nonadjacent to $x$. But now $\{ z, v_1, y, r,  x\}$ induces either a bull or a house, a contradiction. So (\ref{eq:y-dom-x-in-K}) holds.
\begin{equation}\label{eq:W-neighbors-in-R}
	\text{\parbox{.85\textwidth}{
     Any vertex $w\in W$ must have a neighbor in $R$.
    }}
\end{equation}

Suppose that (\ref{eq:W-neighbors-in-R}) were false. Let $w \in W$ be a vertex with no neighbors in $R$. If $R$ is empty, then $v_1\in K$ is a universal vertex, and so $V(G) \setminus v_1$ is a homogeneous set. So we assume $R$ is nonempty. Let $T\subseteq W$ be the set of vertices that have a neighbor in $R$ and let $K^{'}\subseteq K$ be the set of vertices that have a neighbor in $T$. Clearly, $w \in W\setminus T$ and by (\ref{eq:W-R}), there is no edge with one end point in $T$ and the other end point in $W\setminus T$. By (\ref{eq:y-dom-x-in-K}), we see that every vertex in $W\setminus T$ is complete to $K^{'}$. Hence $(W\setminus T)\cup (K\setminus K^{'}) \cup \{v\}$ is a homogeneous set. Thus (\ref{eq:W-neighbors-in-R}) holds.
\begin{equation}\label{eq:W-connected}
	\text{\parbox{.85\textwidth}{
     $G[W]$ is connected. 
    }}
\end{equation}

Suppose $G[W]$ is disconnected. By  (\ref{eq:K-W-complete}) there is a vertex $w_1 \in W$ complete to $K$. Let $r \in R$ be a neighbor of $w_1$. Let $A\subseteq W$ induce a component in $G[W]$ such that $w_1 \notin A$. Since $w_1$ is complete to $K$, by (\ref{eq:y-dom-x-in-K}), each vertex in $A$ is complete to $K$. Thus, if $|A|\ge 2$, then by (\ref{eq:W-R}), $A$ is a homogeneous set, a contradiction. Now we may suppose $A = \{a\}$. If $a$ has no neighbor in $R$, then $\{a, v\}$ is a homogeneous set, a contradiction. If $a$ has a neighbor in $R$, then $G[N(a)]$ is disconnected and from Lemma \ref{lem:split-vertex}, $G$ has a homogeneous set or $G$ is triangle-free, a contradiction. So (\ref{eq:W-connected}) holds. 

If every vertex in $W$ is complete to $K$, then, by (\ref{eq:K-big}), $K$ is a homogeneous set. Therefore, there is a vertex in $W$ that is not adjacent to a vertex in $K$. Since $G[W]$ is connected and $w_1$ is complete to $K$, we must have two adjacent vertices $w$, $w' \in W$ and vertex $u \in K$ such that $wu \in E(G)$ and $w'u \notin E(G)$. Let $R_1$ be the set of vertices in $R$ that have neighbors in $W$ and let $R_2 = R \setminus R_1$. By (\ref{eq:W-R}), every vertex in $W$ is adjacent to all of $R_1$ and none of $R_2$. 

Suppose that $R_2$ is not empty. Since $G$ is connected, there are vertices $r_1 \in R_1, r_2 \in R_2$ with $r_1 r_2 \in E(G)$. Now the set $\{  w', w, r_1, r_2,  u\}$ induces a bull, a contradiction. So we know $R_2 = \emptyset$. If $|R_1| \geq 2$, then $R_1$ is a homogeneous set. So we conclude $R_1$ contains a single vertex, say $r_1$. Recall $v'$ is the simplicial vertex of $G$ that is nonadjacent to $v$. Since the vertices in $W$ are not simplicial (by (\ref{eq:W-neighbors-in-R})), we must have $v' = r_1$. Thus $W$ is a clique. Thus, $G$ is co-bipartite with the partite sets $K \cup \{v\}$ and $W \cup R$. \qed 

We denote by $3K_1$ the stable set on three vertices.
\begin{corollary}\label{cor:bull-P5-free}
    Let $G$ be a connected (bull, $P_5$)-free graph, then
    \begin{itemize}[noitemsep, topsep=0pt]
        \item[$(i)$] $G$ contains a homogeneous set, or
        \item[$(ii)$] $G$ is 3$K_1$-free, or
        \item[$(iii)$] $G$ is bipartite.
    \end{itemize}
\end{corollary}

\section{Critical (bull, $P_5$)-free graphs}\label{sec:critical}
In this section, we prove Theorem~\ref{thm:critical}. We will need to discuss the idea of modular decomposition. 
A set $M \subseteq V(G)$ is a {\it module} if every vertex in $V(G) \setminus M$ is complete or anti-complete to $M$. A module $M$ is {\it trivial} if $M = \emptyset, |M| =1$ or $M = V(G)$. Thus, non-trivial modules are homogeneous sets.  A module $M$ is {\it strong} if it does not overlap another module. Module $M$ is {\it maximal}  if $M \not= V(G)$ and there is no module $M' \not= V(G)$ with $M \subset M'$. Gallai \cite{Gal1967} showed that if both $G$ and $\overline{G}$ are connected then all maximal modules are strong (see \cite{MafPre2001} for an English translation of Gallai's paper). In this case, the maximal modules of $G$ form a partition of $V(G)$. The {\it join} of two graphs $A,B$ is the graph obtained by taking $A$, $B$ and join each vertex of $A$ to each vertex of $B$ by an edge (so $V(A)$ is complete to $V(B)$). 

A graph is {\it prime} if it contains no non-trivial modules (i.e. no homogeneous sets). Let $G$ be a graph such that both $G$ and $\overline{G}$ are connected. The \emph{skeleton} of $G$, denoted $G_{s}$, is the graph obtained by contracting each maximal module to a single vertex (and removing any loops and multiple edges created). Note that $G_{s}$ is a prime subgraph of $G$, induced by a set containing exactly one vertex from each maximal module in $G$.
A \emph{blowup} of a graph $G$ is a graph obtained by substituting non-empty cliques for some vertices of $G$.

For the remainder of this section, we use the following decomposition for graphs $G$ that are connected and 
co-connected.  Let $S_1$, $S_2$,\dots, $S_m$ be a partition of $V(G)$ into maximal modules. Note that $S_p \cap S_q$ = $\emptyset$ for all $p\neq q$, $p$ and $q \in [m]$. Let $\chi(G) = k$ and let $\chi(G[S_p]) = k_p$, for all $p\in [m]$. We define \emph{clique skeleton} of $G$ as the graph obtained from $G$ by recursively substituting a clique $Q_p$ of size $k_p$ for each maximal module $S_p$ in $G$, where $p \in [m]$. We use $G(Q_1, Q_2,\dots,Q_m)$ to denote the clique skeleton of $G$. Note that the clique skeleton of a graph $G$, say $H$, is a blowup of the skeleton of $G$ and $\chi(G)$ = $\chi(H)$. \\

The result below can be obtained from Lemma 5 in \cite{YuJoo2026}.

\begin{lemma}\label{lem:clq-sklton}
    Suppose $G$ is a $k$-critical graph and co-connected, then the clique skeleton $H$ = $G(Q_1, Q_2,\dots,Q_m)$ is $k$-critical.
\end{lemma}
\begin{proof}
    Assume the hypothesis. For a $k$-critical graph $G$ it is not difficult to see that for each module $S$, $G[S]$ must be $\chi(G[S])$-critical. Thus, recursively substituting a clique of size $\chi(G[S])$, which is also $\chi(G[S])$-critical, for each maximal module $S$ in $G$ results in a $k$-critical graph.
\end{proof}

 A \emph{$k$-critical blowup} of a graph $G$ is a blowup of $G$ which is $k$-critical. Note that for some graphs $G$ and for some positive integers $k$, for example $C_4$ and $k$ = 3, there is no $k$-critical blowup of $G$.

\begin{theorem}\label{thm:clq-sklton}
    Suppose $\mathcal{G}$ is a hereditary class of graphs. Let $k\ge 1$ be a positive integer. Suppose, for all positive integer $1\le i \le k$, if the number of vertices in any $i$-critical blowup of every prime graph in $\mathcal{G}$ is bounded by a constant, then the number of vertices in each $j$-critical graph in $\mathcal{G}$ is bounded by a constant, where $1\le j\le k$.
\end{theorem}
\begin{proof}
    Assume the hypotheses. The proof is by contradiction. Suppose the Theorem was false. Let  $j\le k$ be the smallest integer such that there exists a $j$-critical graph $G \in \mathcal{G}$ whose order is not bounded by a constant. Then obviously $G$ is connected. Let $\mathcal{C}(\ell)$ denote the maximum number of vertices in an $\ell$-critical graph in $\mathcal{G}$. Suppose $G$ is a join of two graphs $G_1$ and $G_2$, then clearly $G_1$ is $j_1$-critical and $G_2$ is $j_2$-critical for some positive integers $j_1$, $j_2$ such that $j_1$ + $j_2$ = $j$. Then $|V(G)|\le 2\times max\{\mathcal{C}(\ell)\mid 1\le \ell \le j-1\}$, and so $|V(G)|$ is bounded, a contradiction.

    So we assume $G$ is not a join of two graphs, that is, $G$ is co-connected. Then $V(G)$ can be uniquely partitioned into maximal modules $S_1$, $S_2$,\dots, $S_m$. Let $H$ = $G(Q_1, Q_2,\dots,Q_m)$ be the clique skeleton of $G$ and $G_s$ be the skeleton of $G$. Since $H$ is a blowup of the prime graph $G_s$ in $\mathcal{G}$, from Lemma \ref{lem:clq-sklton} and from our hypothesis, $|V(H)|$ is bounded and hence the number of maximal modules in $H$, say $m$, is also bounded. Note that $G$ also has $m$ maximal modules. Since each module in $G$ is $b$-critical for some $b\le j-1$, we obtain that $|V(G)| \le m\times max\{\mathcal{C}(\ell)\mid 1\le \ell \le j-1\}$, and so $|V(G)|$ is bounded, a contradiction.
\end{proof}

Note that a blowup of a perfect graph is perfect and thus, from Theorem \ref{thm:clq-sklton}, to prove that the number of $k$-critical graphs in $\mathcal{G}$ is finite, it is sufficient to prove that the number of $k$-critical blowup of non perfect prime graphs in $\mathcal{G}$ is finite.

\noindent {\it Proof of Theorem \ref{thm:critical}.}
    From Theorem \ref{thm:clq-sklton}, it is sufficient to prove that the number of vertices in each $k$-critical blowup of every prime ($P_5$, bull)-free graph is bounded by a constant. From Corollary \ref{cor:bull-P5-free}, every prime ($P_5$, bull)-free graph is either bipartite or 3$K_1$-free. The only $k$-critical blowup of any prime bipartite graph is the clique on $k$ vertices. Since the blowup of any 3$K_1$-free graph is 3$K_1$-free, the bound on the number of vertices in a $k$-critical blowup of any prime 3$K_1$-free graph is bounded by the Ramsey number $R(3,k+1)$. \qed

\section{Perfect divisibility of ($P_5$, bull)-free graphs}\label{sec:divisibility}
Chudnovsky and Sivaraman \cite{ChuSiv2019} proved the following theorem. 
\begin{theorem}[Theorem 3.7 in \cite{ChuSiv2019}]\label{thm:perfectlydivisible} 
($P_5$, bull)-free graphs are perfectly divisible. 
\end{theorem}
Using our structural result we will give a short proof of this theorem. Let us call a graph $G$ {\it minimally non-perfectly divisible} (MNPD, for short) if $G$ is not perfectly divisible but every proper induced subgraph of $G$ is. 
\begin{lemma}[Lemma 7.2 in \cite{Hoa2018}]\label{lem:homogeneous}
    A MNPD graph cannot contain a homogeneous set.
\end{lemma}
\begin{lemma}[Lemma 7.3 in \cite{Hoa2018}]\label{lem:3K1}
    A MNPD graph must contain a $3K_1$.
\end{lemma}
\noindent {\it Proof of Theorem \ref{thm:perfectlydivisible}}. We will prove by contradiction. 
Let $G$ be a ($P_5$, bull)-free graph. Suppose that $G$ is not perfectly divisible. Then $G$ must contain an induced MNPD graph. Thus we may assume $G$ is MNPD. By Lemma \ref{lem:homogeneous}, $G$ cannot contain a homogeneous set. By Lemma \ref{lem:3K1}, $G$ contains a $3K_1$. By Corollary \ref{cor:bull-P5-free}, $G$ is bipartite. But bipartite graphs are perfect and therefore perfectly divisible. \qed

We note the same proof shows a slightly stronger statement:  ($P_5$, bull)-free graphs are perfectly divisible for all weight functions (see \cite{Hoa2018} for the definition of weight functions.)

\section{Conclusions}\label{sec:conclusions}
Both the classes of $P_5$-free graphs and bull-free graphs are being intensively studied. It is known that the number of $k$-critical $P_5$-free graphs, for $ k \geq 5$, is infinite. In this paper, we showed that the number of $k$-critical ($P_5$, bull)-free graphs is finite for every $k$. Recently Yu, Jooken, Goedgebeur, and Huang \cite{YuJoo2026} proved that the number of 5-critical ($P_6$, bull)-free graphs is finite.  It would be interesting to solve the Finiteness Problem for 6-critical ($P_6$, bull)-free graphs.

\section{Acknowledgement}
 The author M. B. is supported by ANRF-NPDF, Govt. of India (No. PDF/2025/005208). ORCID: 0000-0002-3153-2339. The author C. T. H. is supported by individual NSERC Discovery Development Grant no. DDG-2024-00015. The majority of this work was done during when the author M. B. was a postdoc at Wilfrid Laurier University, the position was supported by the Natural Sciences and Engineering Research Council of Canada Grant to the author C. T. H. Part of this work was done when the author C. T. H. visited the Indian Statistical Institute in Chennai and thank T. Karthick for the hospitality. The authors thank Ben Cameron and Iain Beaton for pointing out an error in a previous version of this paper.

\end{document}